\documentclass[12pt,a4paper]{amsart}
\usepackage[headings]{fullpage}
\usepackage{amssymb}
\usepackage{amsthm}
\usepackage{listings}
\usepackage{graphicx,float}
\usepackage[usenames]{color}
\usepackage[normalem]{ulem}

\definecolor{orange}{RGB}{255,127,0}

\newcommand{\Bc}{\check{B}}
\newcommand{\Pol}{\operatorname{Pol}}
\newcommand{\Hom}{\operatorname{Hom}}
\newcommand{\Gal}{\operatorname{Gal}}
\newcommand{\Int}{\operatorname{Int}}
\newcommand{\lc}{\operatorname{lc}}
\newcommand{\cK}{\mathcal{K}}
\newcommand{\Qp}{\mathbf{Q}_p}
\newcommand{\Qpbar}{\overline{\mathbf{Q}}_p}
\newcommand{\Cp}{\mathbf{C}_p}

\newcommand{\QQ}{\mathbf{Q}}
\newcommand{\NN}{\mathbf{N}}
\newcommand{\Gm}{\mathbf{G}_\mathrm{m}}
 
\renewcommand{\phi}{\varphi} 
\newcommand{\frX}{\mathfrak{X}}
\newcommand{\MM}{\mathfrak{m}}
\newcommand{\val}{\operatorname{val}}
\newcommand{\vp}{\val_p}
\newcommand{\Tr}{\operatorname{Tr}}

\newcommand{\cyc}{\operatorname{cyc}}
\newcommand{\dcroc}[1]{[\![ #1 ]\!]}
\newcommand{\LT}{\operatorname{LT}}

\theoremstyle{plain}

\newtheorem{theo}{Theorem}[section]
\newtheorem{coro}[theo]{Corollary}
\newtheorem{lemm}[theo]{Lemma}
\newtheorem{prop}[theo]{Proposition}

\newtheorem{defi}[theo]{Definition}

\newtheorem{rema}[theo]{Remark}

\newtheorem*{theoA}{Theorem A}
\newtheorem*{defiA}{Definition}
\newtheorem*{theoB}{Theorem B}
\newtheorem*{quesA}{Question}

\begin{document}

\title{Integer-valued polynomials and $p$-adic Fourier theory}

\author{Laurent Berger}
\address{UMPA, ENS de Lyon \\ 
UMR 5669 du CNRS \\
Lyon, France}
\email{laurent.berger@ens-lyon.fr}
\urladdr{https://perso.ens-lyon.fr/laurent.berger/}

\author{Johannes Sprang}
\address{Fakult\"at f\"ur Mathematik \\
Universit\"at Duisburg-Essen \\
Essen, Germany}
\email{johannes.sprang@uni-due.de}
\urladdr{https://www.esaga.uni-due.de/johannes.sprang/}

\begin{abstract}
The goal of this paper is to give a numerical criterion for an open question in $p$-adic Fourier theory. Let $F$ be a finite extension of $\mathbf{Q}_p$. Schneider and Teitelbaum defined and studied the character variety $\mathfrak{X}$, which is a rigid analytic curve over $F$ that parameterizes the set of locally $F$-analytic characters $\lambda : (o_F,+) \to (\mathbf{C}_p^\times,\times)$. Determining the structure of the ring $\Lambda_F(\mathfrak{X})$ of bounded-by-one functions on $\mathfrak{X}$ defined over $F$ seems like a difficult question. Using the Katz isomorphism,  we prove that if $F= \mathbf{Q}_{p^2}$, then $\Lambda_F(\mathfrak{X}) = o_F [\![o_F]\!]$ if and only if the $o_F$-module of integer-valued polynomials on $o_F$ is generated by a certain explicit set. Some computations in SageMath indicate that this seems to be the case.
\end{abstract}

\date{\today}

\maketitle

\tableofcontents

\setlength{\baselineskip}{17pt}

\section{Introduction}
\label{intro}

Let $F$ be a finite extension of $\Qp$ with ring of integers $o_F$. Let $\Cp$ denote the completion of an algebraic closure of $\Qp$. In their work on $p$-adic Fourier theory \cite{ST01}, Schneider and Teitelbaum defined and studied the character variety $\frX$. This character variety is a rigid analytic curve over $F$ that parameterizes the set of locally $F$-analytic characters $\lambda : (o_F,+) \to (\Cp^\times,\times)$. Let $\Lambda_F(\frX)$ denote the ring of functions on $\frX$ defined over $F$ whose norms are bounded above by $1$. If $\mu \in o_F \dcroc{o_F}$ is a measure on $o_F$, then $\lambda \mapsto \mu(\lambda)$ gives rise to such a function $\frX \to \Cp$. The resulting map $o_F \dcroc{o_F} \to \Lambda_F(\frX)$ is injective. We do not know of any example of an element of $\Lambda_F(\frX)$ that is not in the image of the above map. 

\begin{quesA}
\label{quesA}
Do we have $\Lambda_F(\frX) = o_F \dcroc{o_F}$?
\end{quesA}

This question seems to be quite difficult, and is extensively studied in \cite{AB24}. The goal of our paper is to give, for $F=\QQ_{p^2}$, a simple criterion for the above question, that can be checked numerically. Numerical evidence then seems to indicate that the answer to the question is ``yes'' for $F=\QQ_{p^2}$. We now formulate this criterion. For the time being, let $F \neq \Qp$ be any finite proper extension of $\Qp$, let $\pi$ be a uniformizer of $o_F$ and let $\LT$ denote the Lubin--Tate formal group attached to $\pi$. Once we have chosen a coordinate $X$ on $\LT$, we have a formal addition law $T\oplus U  \in o_F \dcroc{T,U}$ and endomorphisms $[a](X) \in o_F \dcroc{X}$ for all $a \in o_F$.

Let $\Int$ denote the set of integer-valued polynomials on $o_F$, namely those polynomials $P(T) \in F[T]$ such that $P(o_F) \subset o_F$. If $f(X) \in o_F \dcroc{X}$, there exist polynomials $c_{f,n}(T) \in \Int$ for all $n \geq 0$ such that $f([a](X)) = \sum_{n \geq 0} c_{f,n}(a) X^n$. 

\begin{defiA}
\label{defiA}
If $M$ is a subset of $o_F \dcroc{X}$, let $\Pol(M)$ denote the sub $o_F$-module of $\Int$ generated by the $c_{f,n}$ with $f \in M$ and $n \geq 0$. 
\end{defiA}

In particular, the module $\Pol$ defined in \S 1.5 of \cite{AB24} is equal to $\Pol( \{ 1, X,X^2,\hdots\}) = \Pol(o_F \dcroc{X})$, and theorem 1.5.1 of \cite{AB24} states that if $\Lambda_F(\frX) = o_F \dcroc{o_F}$, then $\Pol=\Int$. However, the reverse implication is not true, and the goal of our paper is to provide an analogous ``if and only if'' statement, at least when $F=\QQ_{p^2}$. The ring $o_F \dcroc{X}$ is equipped with an operator $\phi$ defined by $\phi(f)(X) = f([\pi](X))$ and an operator $\psi$ given by 
\[ \phi(\psi(f(X))) = \frac{1}{\pi} \cdot \Tr_{o_F \dcroc{X} / \phi(o_F \dcroc{X}) } (f(X)). \] 
Note that $\psi(f) = 0$ if and only if $\sum_{\eta \in \MM_{\Cp},[\pi](\eta)=0} f(X \oplus \eta) = 0$. 

The first result of this paper is the following criterion for the above question.

\begin{theoA}
\label{theoA}
If $F=\QQ_{p^2}$, then $\Lambda_F(\frX) = o_F \dcroc{o_F} \Leftrightarrow \Pol(o_F \dcroc{X}^{\psi=0}) = \Int$.
\end{theoA}

Let us remark that the question of whether $\Pol(o_F \dcroc{X}^{\psi=0}) = \Int$ depends neither on the choice of a coordinate on $\LT$ nor on the choice of the uniformizer $\pi$ used in the definition of $\LT$ (for instance, it is equivalent to $\Lambda_F(\frX) = o_F \dcroc{o_F}$ by the above theorem).
The appendix of \cite{AB24} is devoted to checking numerically that $\Pol=\Int$ for various fields $F$. We adapt these methods and give numerical evidence that $\Pol(o_F \dcroc{X}^{\psi=0}) = \Int$ when $F=\QQ_{p^2}$. This is the first compelling evidence in favor of the fact that $\Lambda_F(\frX) = o_F \dcroc{o_F}$, at least for $F=\QQ_{p^2}$. In addition, we prove theorem B below, which implies that $\Pol(o_F \dcroc{X}^{\psi=0})$ is $p$-adically dense in $\Int$. 

\begin{theoB}
\label{theoB}
For all $F$,  we have $\Pol(o_F \dcroc{X}^{\psi=0}) + \pi \cdot \Int = \Int$.
\end{theoB}

The main ingedient for the proof of theorem A is the ``Katz isomorphism'' proved in \cite{AB24} for $F=\QQ_{p^2}$, which gives rise to an isomorphism $\Hom_{o_F}(o_\infty,o_F) \simeq o_F \dcroc{X}^{\psi=0}$ where $o_\infty$ is (at least when $\pi=p$) the ring of integers of the field generated by the torsion points of $\LT$. We prove theorem B by using some results of \cite{SI09} on Mahler bases and coefficients of Lubin--Tate power series. Using these results, it is enough to show that $\Pol(o_F \dcroc{X}^{\psi=0}) + \pi \cdot \Int$ is stable under multiplication. In order to do this, we prove that in some sense, the coefficients of a power series $F(X,Y) \in k \dcroc{X,Y}$ (where $k$ is the residue field of $o_F$) can be recovered from the coefficients of $F(X,[b](X))$ for sufficiently many $b \in o_F$. We also sketch a completely different proof of theorem B, based on a similar unpublished argument of Ardakov for proving that $\Pol(o_F \dcroc{X}) + \pi \cdot \Int = \Int$.

\vspace{\baselineskip}
\noindent\textbf{Acknowledgements.} We would like to thank Konstantin Ardakov, Sandra Rozensztajn and Rustam Steingart for useful comments on a first draft of this paper. 

\section{Notation}
\label{secpol}

We use the notation of the introduction: $F$ is a finite extension of $\Qp$ of degree $d > 1$ and ramification index $e$, with ring of integers $o_F$. The  residue field $k$ of $o_F$ has cardinality $q = p^f$ and $\pi$ is a uniformizer of $o_F$. Let $\Int$ denote the set of integer-valued polynomials on $o_F$, namely those polynomials $P(T) \in F[T]$ such that $P(o_F) \subset o_F$. Let $\LT$ denote the Lubin--Tate formal group attached to $\pi$ (see \cite{LT65}) and let $X$ be a coordinate on $\LT$. We have a formal addition law $T \oplus U  \in o_F \dcroc{T,U}$, endomorphisms $[a](X) \in o_F \dcroc{X}$ for all $a \in o_F$, a logarithm $\log_{\LT}(X) \in F \dcroc{X}$ and a Lubin--Tate character $\chi_\pi : \Gal(\Qpbar/F) \to o_F^\times$. Let $\chi_{\cyc}$ denote the cyclotomic character, and let $\tau : G_L \to o_F^\times$ denote the character $\tau = \chi_{\cyc} \cdot \chi_\pi^{-1}$. If $F \neq \Qp$, the image of $\tau$ is open in $o_F^\times$, compare \cite[Lemma 2.6.3]{AB24}.

The monoid $(o_F,\times)$ acts on $o_F \dcroc{X}$ by $a \cdot f(X) = f([a](X))$. The map $\phi$ is defined by $\phi(f)(X) = f([\pi](X))$ and  $\psi$ is given by $\phi(\psi(f(X))) = 1/\pi \cdot \Tr_{o_F \dcroc{X} / \phi(o_F \dcroc{X}) } (f(X))$.

If $f(X) \in o_F \dcroc{X}$, there exist polynomials $c_{f,n}(T) \in \Int$ for all $n \geq 0$ such that for $a \in o_F$, $f([a](X)) = \sum_{n \geq 0} c_{f,n}(a) X^n$. If $M$ is a subset of $o_F \dcroc{X}$, let $\Pol(M)$ denote the sub $o_F$-module of $\Int$ generated by the $c_{f,n}$ with $f \in M$ and $n \geq 0$. If $i \geq 0$, we let $c_{i,n} = c_{f,n}$ with $f(X) = X^i$. Note that $\Pol( \{ 1, X,X^2,\hdots\}) = \Pol(o_F \dcroc{X})$.

\section{$p$-adic Fourier theory}
\label{fousec}

Recall (see \S 3 and \S 4 of \cite{ST01} for what follows) that $\Hom_{o_{\Cp}} (\LT,\Gm)$ is a free $o_F$-module of rank $1$. Choosing a generator of this module gives a power series $G(X) \in X \cdot o_{\Cp}\dcroc{X}$ such that $G(X) = \Omega \cdot X + \cdots$, where $\Omega \in o_{\Cp}$ is such that $g(\Omega) = \tau(g) \cdot \Omega$ if $g \in \Gal(\Qpbar/F)$ and $\vp(\Omega) = 1/(p-1)-1/e(q-1)$. In particular, $1 + G(X) = \exp(\Omega \cdot \log_{\LT}(X)) = \sum_{n \geq 0} P_n(\Omega) X^n$ where $P_n(Y) \in F[Y]$ is a polynomial of degree $n$ such that $P_n(\Omega \cdot o_F) \subset o_{\Cp}$. 

Let $F_\infty = \Cp^{\ker \tau}$ and let $o_\infty$ denote the ring of integers of $F_\infty$. Note that by \S 2.7 of \cite{AB24}, we have $F_\infty = \widehat{F(\Omega)}$. We have (see \S 3.3 of \cite{AB24} as well as \cite{Kat77}) a map $\cK_1^* : \Hom_{o_F}(o_\infty,o_F) \to o_F \dcroc{X}$ that sends $h \in \Hom_{o_F}(o_\infty,o_F)$ to $\sum_{n \geq 0} h(P_n(\Omega)) X^n$.

\begin{theo}
\label{katzpair}
The map $\cK_1^*$ is injective, its image is included in $o_F \dcroc{X}^{\psi=0}$, and if $F=\QQ_{p^2}$ then it gives rise to an isomorphism $\Hom_{o_F}(o_\infty,o_F) \to o_F \dcroc{X}^{\psi=0}$.
\end{theo}

\begin{proof}
If $ h \in \Hom_{o_F}(o_\infty,o_F)$, then $h$ extends to a continuous $F$-linear map $h : F_\infty \to F$. If $h(P_n(\Omega)) = 0$ for all $n \geq 0$, then $h=0$ on $F[\Omega]$. By prop 6.2 of \cite{APZ98}, $F[\Omega]$ is $p$-adically dense in $\widehat{F(\Omega)} = F_\infty$. This proves the injectivity of $\cK_1^*$. The fact that the image of $\cK_1^*$ is included in $o_F \dcroc{X}^{\psi=0}$ is lemma 3.3.8 of \cite{AB24}. The last assertion is theorem 3.6.14 of \cite{AB24}.
\end{proof}

We now assume that $F = \QQ_{p^2}$. We use the map $\cK_1^*$ to define a pairing $\langle \cdot , \cdot \rangle : o_\infty \times o_F \dcroc{X}^{\psi=0} \to o_F$, given by the formula $\langle z , f(X) \rangle = h(z)$ where $h \in \Hom_{o_F}(o_\infty,o_F)$ is such that $\cK_1^*(h) = f(X)$. By definition, we have $\langle P_n(\Omega) , \sum_{i \geq 0} f_i X^i \rangle = f_n$.

\begin{lemm}
\label{paireq}
If $P(T) \in F[T]$ is such that $P(\Omega) \in o_\infty$, and $f(X) \in o_F \dcroc{X}^{\psi=0}$, then $\langle P(\Omega) , f([a](X)) \rangle = \langle P(a \Omega) , f(X) \rangle$.
\end{lemm}

\begin{proof}
See \S 3.2 of \cite{AB24}, in particular equation (3) above definition 3.2.4.
\end{proof}

Let $B = F[\Omega] \cap o_\infty$ and pick a regular basis (definition 4.2.5 of \cite{AB24}) $\{ b_n(\Omega) \}_{n \geq 0}$ for $B$. Recall (lemma 4.2.8 of \cite{AB24}) that the polynomials $\rho_{i,k}(T) \in \Int$ are defined by $P_k(a \Omega) = \sum_{i=0}^k \rho_{i,k}(a) b_i(\Omega)$. As in \S 4.2 of \cite{AB24}, let $\Bc \subset \Int$ denote the $o_F$-span of the $\rho_{i,k}(T)$ with $i,k \geq 0$. We then have (corollary 4.2.19 of \cite{AB24}) the following criterion.

\begin{prop}
\label{bcrit}
We have $\Lambda_F(\frX) =o_F \dcroc{o_F}$ if and only if $\Bc = \Int$.
\end{prop}

Given this criterion, theorem A results from the following claim.

\begin{prop}
\label{cfnspan}
If $F = \QQ_{p^2}$, then $\Pol(o_F \dcroc{X}^{\psi=0}) = \Bc$.
\end{prop}

\begin{proof}
Recall that if $f(X) \in o_F \dcroc{X}^{\psi=0}$, we write $f([a](X)) = \sum_{n \geq 0} c_{f,n}(a) X^n$. We first prove that each $c_{f,n}$ is in $\Bc$. By lemma \ref{paireq}, $\langle P_k(\Omega) , f([a](X)) \rangle = \langle P_k(a \Omega) , f(X) \rangle$, so that 
\begin{multline*}
c_{f,k}(a) = \langle P_k(\Omega) , f([a](X)) \rangle 
= \langle P_k(a \Omega) , f(X) \rangle \\ 
= \langle \sum_{i=0}^k \rho_{ik}(a) b_i(\Omega), f(X) \rangle 
= \sum_{i=0}^k \rho_{ik}(a) \langle b_i(\Omega), f(X) \rangle,
\end{multline*}
and each $\langle b_i(\Omega), f(X) \rangle$ belongs to $o_F$ since $b_i(\Omega) \in o_\infty$. Hence $\Pol(o_F \dcroc{X}^{\psi=0}) \subset \Bc$.

To show equality, the above computation implies that it is enough to show that given $k \geq 0$ and $j \leq k$, there exists $f(X) \in o_F \dcroc{X}^{\psi=0}$ with $\langle b_i(\Omega), f(X) \rangle = \delta_{i,j}$ for $0 \leq i \leq k$. 

Let $N = o_F \cdot b_0(\Omega) + \cdots + o_F \cdot b_k(\Omega) = F[\Omega]_k \cap o_\infty$. The $o_F$-module $N$ is a finitely generated and pure submodule of the $o_F$-module $o_\infty$, hence a direct summand (see \S 16 of \cite{Kap69}, in particular exercise 57). The map $N \to o_F$ that sends $b_i(\Omega)$ to $\delta_{i,j}$ therefore extends to an $o_F$-linear map $h : o_\infty \to o_F$. We can now take $f(X) = \cK_1^*(h)$.
\end{proof}

\begin{rema}
\label{beyond}
If $F \neq \QQ_{p^2}$, let $M$ denote the image of the map $\cK_1^* : \Hom_{o_F}(o_\infty,o_F) \to o_F \dcroc{X}^{\psi=0}$. The proof of prop \ref{cfnspan} shows that $\Pol(M) = \Bc$ and hence that $\Lambda_F(\frX) =o_F \dcroc{o_F}$ if and only if $\Pol(M) = \Int$. It would therefore be interesting to compute $M$ in general. Another consequence of this is that if $\Lambda_F(\frX) =o_F \dcroc{o_F}$, then $\Pol(o_F \dcroc{X}^{\psi=0}) = \Int$.
\end{rema}

\section{Pol and Int modulo $\pi$}

In this {\S}, we prove theorem B. Let $F$ be a finite extension of $\Qp$. Recall that $k$ denotes the residue field of $o_F$. Let $B$ be a finite subset of $o_F$ and let $s : k \dcroc{X,Y} \to \prod_{b \in B} k \dcroc{X}$ be the substitution map $F(X,Y) \mapsto \{ F(X,[b](X)) \}_{b \in B}$. 

\begin{lemm}
\label{kersb}
We have $\ker s = \prod_{b \in B} (Y-[b](X)) \cdot  k \dcroc{X,Y}$. 
\end{lemm}

\begin{proof}
If $F(X,[b](X)) = 0$, then $F(X,Y) = F(X,Y)  - F(X,[b](X)) = (Y - [b](X)) \cdot G(X,Y)$. This implies the claim by induction since $[b](X) \neq [b'](X)$ if $b \neq b'$.
\end{proof}

Let $d = |B|$ and let $I=(X,Y)^d$ so that $I$ is an open neighborhood of $\ker s$ in $k \dcroc{X,Y}$ and we have a well-defined and injective map $s : k \dcroc{X,Y} / I \to k \dcroc{X}^d / s(I)$. 

\begin{lemm}
\label{nfork}
There exists $n=n(B)$ having the property that if $f \in k \dcroc{X,Y}$ is such that $f(X,[b](X)) \in X^n k \dcroc{X}$ for all $b \in B$, then $f \in I$.
\end{lemm}

\begin{proof}
If there is no such $n$, then for all $n$ there is an $f_n$ contradicting the lemma. Since $k \dcroc{X,Y} / I$ is a finite set, there is an $f$ not in $I$ such that $f(X,[b](X)) \in X^n k \dcroc{X}$ for all $b \in B$ and infinitely many $n$, so that $f(X,[b](X)) = 0$ for all $b \in B$. Hence $f \in \ker s \subset I$ by lemma \ref{kersb}.
\end{proof}

\begin{coro}
\label{injevt}
The map $s : k \dcroc{X,Y} / I \to k \dcroc{X}^d / (X^n + s(I))$ is injective.
\end{coro}

\begin{proof}
If $f \in k \dcroc{X,Y}$ is such that $s(f) = X^n g + s(i) \in X^n k \dcroc{X}^d + s(I)$, the above lemma applied to $f-i$ shows that $f-i \in I$.
\end{proof}

If $h(X) \in k \dcroc{X}$, let $\langle h(X) \vert X^j \rangle \in k$ denote the coefficient of $X^j$ in $h(X)$. If $d \geq 1$ and  $m+\ell < d$, the coefficient of $X^m Y^\ell$ in $F(X,Y) \in k \dcroc{X,Y} / I$ is well defined.

\begin{prop}
\label{coeff}
If $d \geq 1$, and $B=\{b_1,\hdots,b_d\}$ and $n=n(B)$ is as above, and if $m+\ell < d$, there exist some $\mu_{i,j} \in k$ for $1 \leq i \leq d$ and $0 \leq j \leq n-1$, such that for all $F(X,Y) \in k \dcroc{X,Y}$, the coefficient of $X^m Y^\ell$ in $F(X,Y)$ is equal to \[ \sum_{i=1}^d \sum_{j=0}^{n-1} \mu_{i,j} \cdot \langle F(X,[b_i](X)) \vert X^j \rangle. \]
\end{prop}

\begin{proof}
Let $M = k \dcroc{X,Y} / I$ and $N = k \dcroc{X}^d / (X^n + s(I))$; they are both finite dimensional $k$-vector spaces. Let $h_{m,\ell} : M \to k$ be the linear form giving the coefficient of $X^m Y^\ell$ in $F(X,Y) \bmod{I}$. Consider the injective map (lemma \ref{injevt}) $s : M \to N$. The linear form $h_{m,\ell} \circ s^{-1} : s(M) \to k$ extends to a linear form $\lambda : N \to k$ which in turn gives rise to a linear form $\mu :  k \dcroc{X}^d / X^n \to k$ factoring through $N$. 

There exist some $\mu_{i,j} \in k$ such that if $f=(f_1,\hdots,f_d) \in k \dcroc{X}^d / X^n$, then $\mu(f) = \sum_{i=1}^d \sum_{j=0}^{n-1} \mu_{i,j} \cdot \langle f_i(X) \vert X^j \rangle$. If $F \in k \dcroc{X,Y} / I$, we have 
\begin{align*} 
h_{m,\ell}(F) & = \lambda \circ s (F) \\
& =  \mu\left( F(X,[b_1](X)),\hdots,F(X,[b_d](X)) \right) \\
& = \sum_{i=1}^d \sum_{j=0}^{n-1} \mu_{i,j} \cdot \langle F(X,[b_i](X)) \vert X^j \rangle. \qedhere 
\end{align*}
\end{proof}

\begin{lemm}
\label{psibpi}
If $f(X) \in o_F \dcroc{X}^{\psi=0}$, $g(X) \in o_F \dcroc{X}$ and $b \in \pi \cdot o_F$, then $f(X)g([b](X)) \in o_F \dcroc{X}^{\psi=0}$.
\end{lemm}

\begin{proof}
This follows from
\[ \sum_{\eta \in \MM_{\Cp},[\pi](\eta)=0} f(X \oplus \eta)g([b](X \oplus \eta)) = g([b](X))\sum_{\eta \in \MM_{\Cp},[\pi](\eta)=0} f(X \oplus \eta)=0.
\qedhere \]
\end{proof}

\begin{theo}
\label{polring}
If $f,g \in o_F \dcroc{X}^{\psi=0}$ and $m,\ell \geq 0$ then $c_{f,m}(T) \cdot c_{g,\ell}(T) \in \Pol(o_F \dcroc{X}^{\psi=0})/\pi$.
\end{theo}

\begin{proof}
If $a \in o_F$, then $c_{f,m}(a) \cdot c_{g,\ell}(a)$ is the coefficient of $X^m Y^\ell$ in $f([a](X)) \cdot g([a](Y))$. Choose $d > m+\ell$ and $B \subset \pi \cdot o_F$ with $|B| = d$.

If $H(X,Y) \in k \dcroc{X,Y}$ and $h_b(X) = H(X,[b](X))$ for $b \in o_F$, then
\[ H([a](X),[a]([b](X))) = H([a](X),[b]([a](X))) = h_b([a](X)). \]
Take $H(X,Y) = f(X)\cdot g(Y)$ and let $F(X,Y) = f([a](X)) \cdot g([a](Y)) = H([a](X),[a](Y))$. Since $B \subset \pi \cdot o_F$, lemma \ref{psibpi} implies that $h_b(X) = f(X) \cdot g([b](X))$ belongs to $o_F \dcroc{X}^{\psi=0}$. 
By prop \ref{coeff}, the coefficient of $X^m Y^\ell$ in $F(X,Y)$ is $\sum_{i=1}^d \sum_{j=0}^{n-1} \mu_{i,j} \cdot \langle h_{b_i}([a](X)) \vert X^j \rangle$, and hence $c_{f,m}(T) \cdot c_{g,\ell}(T) = \sum_{i=1}^d \sum_{j=0}^{n-1} \mu_{i,j} \cdot c_{h_{b_i},j}(T) \in \Pol(o_F \dcroc{X}^{\psi=0})/\pi$.
\end{proof}

Recall that $[a](X) = \sum_{n \geq 1} c_{1,n}(a) X^n$ with $c_{1,n}(T) \in \Int$, and that a Mahler basis for $o_F$ is a regular basis of $\Int$.

\begin{lemm}
\label{dsi}
The functions $c_{1,q^k}$ are part of a Mahler basis for $o_F$, and the $o_F$-algebra $\Int$ is generated by the $c_{1,q^k}$ for $k \geq 0$.
\end{lemm}

\begin{proof}
See theorem 3.1 of \cite{SI09}.
\end{proof}

\begin{lemm}
\label{formpsi}
If $X$ is a coordinate on $\LT$ such that $[\pi](X)=\pi X+X^q$, then
\[ \psi(1) = \frac{q}{\pi}, \qquad
\psi(X^i) =0 \text{ if $1 \leq i \leq q-2$}, \qquad
\psi(X^{q-1}) = 1-q, \]
and 
\[ \psi(X^i)=X\cdot \psi(X^{i-q})-\pi\cdot \psi(X^{i-q+1}),\text{ for }i\geq q. \]
\end{lemm}
 
\begin{proof}
See the proof of \cite[Prop 2.2]{FX13}, noting that their $\psi$ is $\pi/q$ times our $\psi$.
\end{proof}

\begin{lemm}
\label{cqkinpsizero}
	For all $k\geq 0$, we have $c_{1,q^k}\in \Pol(o_F \dcroc{X}^{\psi=0})+\pi \cdot \Int$. In fact, we even have $c_{1,q^k}\in  \Pol(o_F \dcroc{X}^{\psi=0})$ if $q\neq 2$ or if $F=\QQ_2$.
\end{lemm}

\begin{proof}
For $q\neq 2$, Lemma \ref{formpsi} shows $\psi(X)=0$ and the claim follows. If $F=\QQ_2$ and $f(X)=X-(1-q)\frac{\pi}{q}$, then Lemma \ref{formpsi} shows $\psi(f)=0$ and the claim follows from $c_{f,q^k}=c_{1,q^k}$. If $q=2$ and $F \neq \QQ_2$, 
we have according to Lemma \ref{formpsi}
\[ \psi(X^2) = \frac{q}{\pi}X-\pi(1-q), \qquad \psi(X^3) = (1-2q)\cdot X + \pi^2(1-q). \]
Hence if $f(X) =X^2-\frac{q}{\pi(1-2q)}X^3+\pi\left(1+\frac{q}{1-2q}\right)X$, then $\psi(f)=0$. Note that $q/\pi \equiv 0 \mod \pi$ so that $f(X) \equiv X^2 \mod \pi$. The residue field of $F$  is $\mathbf{F}_2$, so the claim follows from $c_{f,q^{k+1}}\equiv c_{2,2\cdot q^{k}}\equiv c_{1,q^k}^2\equiv c_{1,q^k} \mod \pi \cdot \Int$.
\end{proof}

\begin{proof}[Proof of theorem B]
Note that $\Pol(o_F \dcroc{X}^{\psi=0})$ is independent of the choice of coordinate $X$ on $\LT$. We choose one such that $[\pi](X)=\pi X+X^q$. By lemma \ref{formpsi}, $1-q/(\pi (1-q)) \cdot X^{q-1}$ belongs to $o_F \dcroc{X}^{\psi=0}$ so that $1 \in \Pol(o_F \dcroc{X}^{\psi=0})$. Lemma \ref{cqkinpsizero} shows that $c_{1,q^k}\in \Pol(o_F \dcroc{X}^{\psi=0})+\pi\cdot \Int$ for all $k\geq 0$. Theorem B now follows from lemma \ref{dsi}, and from theorem \ref{polring}.
\end{proof}

Another proof of theorem B using completely different ideas is sketched below. It is based on arguments of Konstantin Ardakov for proving that $\Pol(o_F \dcroc{X}) + \pi \cdot \Int = \Int$. Let $M$ be either $o_F \dcroc{X}$ or $o_F \dcroc{X}^{\psi=0}$. 

\begin{lemm}
\label{pollin}
If $b \in o_F^\times$, then $\Pol(M)$ is stable under $P(T) \mapsto P(b \cdot T)$.
\end{lemm}

\begin{proof}
We have $c_{f \circ [b],n} (a) = c_{f,n}(ba)$ and if $f \in M$ and $b \in o_F^\times$, then $f \circ [b] \in M$.
\end{proof}

Let $\partial = \log_{\LT}'(X)^{-1} \cdot d/dX$ be the normalized invariant differential on $F \dcroc{X}$. Recall (see \S1 of \cite{Kat81}) that if $f(X) \in F \dcroc{X}$, then $f(X \oplus H) = \sum_{n \geq 0} P_n(\partial)(f(X)) \cdot H^n$. 

\begin{lemm}
\label{poltrn}
If $b \in o_F$, then $\Pol(M)$ is stable under $P(T) \mapsto P(T + b)$.
\end{lemm}

\begin{proof}
We first check that $M$ is stable under $P_n(\partial)$ for all $n \geq 0$. We have $f(X \oplus H) = \sum_{n \geq 0} P_n(\partial)(f(X)) \cdot H^n$ and $P_n(\partial)(f(X))$ belongs to $o_F \dcroc{X}$ as $f(X \oplus H) \in o_F \dcroc{X,H}$. Finally $\partial \circ \psi = \pi^{-1} \psi \circ \partial$ so that if $\psi(f)=0$ then $\psi(P_n(\partial)f)=0$ as well.

If $f(X) \in M$, then $f([a+b](X)) = \sum_{i \geq 0} c_{f,i}(a+b) X^i$. On the other hand,
 \[ f([a+b](X)) = f([a](X) \oplus [b](X)) = \sum_{n \geq 0} P_n(\partial)(f) ([a](X)) \cdot [b](X)^n. \]
 This implies that 
 \[
 c_{f,i}(T+b) = \sum_{\substack{\ell + m = i,\\ 0 \leq n \leq m}} c_{P_n(\partial)(f),\ell}(T) c_{n,m}(b). \qedhere
 \]
\end{proof}

\begin{prop}
\label{polbig}
{\ }
\begin{enumerate}
\item The image of $\Pol(o_F \dcroc{X}^{\psi=0})$ in $\Int / \pi \cdot \Int$ is infinite dimensional.
\item We have $\Pol(o_F \dcroc{X}^{\psi=0})+ \pi \cdot \Int = \Int$.
\end{enumerate}
\end{prop}

\begin{proof}
Lemma \ref{cqkinpsizero} shows that $c_{1,q^k}\in \Pol(o_F \dcroc{X}^{\psi=0})+\pi\cdot \Int$ for all $k\geq 0$. By lemma \ref{dsi}, these elements  are part of a Mahler basis, hence linearly independent mod $\pi$. This implies (1). 

We now sketch the proof of (2). We have $\Int / \pi = C^0(o_F,k)$ and its dual is $k \dcroc{o_F}$. Let $I \subset k \dcroc{o_F}$ be the orthogonal of the image $P$ of $\Pol(o_F \dcroc{X}^{\psi=0})$ in $C^0(o_F,k)$. Since $P$ is stable under $f(T) \mapsto f(b \cdot T)$ for $b \in o_F^\times$ by lemma \ref{pollin}, $I$ is stable under the action of $o_F^\times$. Since $P$ is also stable under $f(T) \mapsto f(T+b)$ for $b \in o_F$ by lemma \ref{poltrn}, $I$ is an ideal of $k \dcroc{o_F}$. By either \S 8.1 of \cite{Ard12} or the main result of \cite{HMS14}, either $I=\{0\}$ or $I$ is open in $k \dcroc{o_F}$. By item (1), $I$ cannot be open, so that $I=\{0\}$ and hence $P = \Int/\pi$. 
\end{proof}

\section{Numerical verification}

In this {\S}, we assume that $F=\QQ_{p^2}$ and that $\pi=p$. We choose a coordinate $X$ on $\LT$ with $[p](X)=pX + X^q$. 

For a non-negative integer $n$, we write $\Int_{\leq n}$ for the $o_F$-module of integer-valued polynomials on $o_F$ of degree $\leq n$. Similarly, let us denote by $\Pol(o_F \dcroc{X}^{\psi=0})_{\leq n}$ the $o_F$-module generated by the set
\[\{ c_{f,m}: 0\leq m\leq n,\, f\in o_F \dcroc{X}^{\psi=0}\}.\] 
Recall from \cite[Proposition 2.2]{SI09}, that a sequence $\{ P_n \}_{n\geq 0}$ of integer-valued polynomials with $\deg(P_n)=n$ is a basis of $\Int$ if and only if $v_p(\lc(P_n))=-w_q(n)$, where 
 $w_q(n) = \sum_{k \geq 1} \lfloor n/q^k \rfloor$
and $\lc(P)$ denotes the leading coefficient of a polynomial $P$.

\begin{defi}
\label{defs0}
For a fixed $n\in \NN$, we define
$s_0(n) = \inf \{ N \geq n$ such that there exists $P\in \Pol(o_F \dcroc{X}^{\psi=0})_{\leq N}$ of degree $n$ such that $v_p(\lc(P))=-w_q(n) \}$.
\end{defi}

By Theorem A, we have  $\Lambda_F(\frX) = o_F \dcroc{o_F}$ if and only if $s_0(n)$ is finite for every $n\in \NN$. In this section, we  explain how to compute $s_0(n)$ numerically. As a first step, let us give explicit generators of the $o_F$-module $\Pol(o_F \dcroc{X}^{\psi=0})_{\leq N}$. 

Recall that $F=\QQ_{p^2}$ and that $\pi=p$.

\begin{lemm}
\label{ofpsizer}
We have $o_F \dcroc{X}^{\psi=0} = (\oplus_{i=1}^{q-2} X^i  \cdot \phi(o_F \dcroc{X}) ) \oplus (pX^{q-1} - (1-q))  \cdot \phi(o_F \dcroc{X})$.
\end{lemm}

\begin{proof}
We have $o_F \dcroc{X} = \oplus_{i=0}^{q-1} X^i \phi(o_F \dcroc{X})$. By lemma \ref{formpsi}, $\psi(1)=p$ and $\psi(X^i) = 0$ for $1 \leq i \leq q-2$ and $\psi(X^{q-1}) = (1-q)$. If $f = \sum_{i=0}^{q-1} X^i \phi(f_i)$ then $\psi(f) = p \cdot f_0 + (1-q) \cdot f_{q-1}$ so that $\psi(f)=0$ if and only if $f_{q-1} = -p/(1-q) \cdot f_0$.
\end{proof}

Since $c_{f,m}$ for $0\leq m\leq N$  only depends on $f\in o_F \dcroc{X}/X^{N+1} \dcroc{X}$, we obtain the following corollary:

\begin{coro}
\label{ofpsizer-coro}
Let $b_0,\dots,b_N$ be a basis of the $o_F$-submodule of $o_F \dcroc{X}/X^{N+1} \dcroc{X}$ generated by the images of
\begin{equation}\label{eq:genpsizer}
	\{ (pX^{q-1} - (1-q)) \phi(X)^j : 0\leq j\leq N  \}\cup \{ X^i \phi(X)^j : 0\leq i+ j\leq N,\  1 \leq i \leq q-2 \}
\end{equation}
in $o_F \dcroc{X}/X^{N+1} \dcroc{X}$, then the $o_F$-module $\Pol(o_F \dcroc{X}^{\psi=0})_{\leq N}$ is generated by
\[
	\{ c_{b,m}: b=b_0,\dots,b_N, 0\leq m\leq N \}.
\]
\end{coro}
The basis $b_0,\dots,b_N$ can easily be computed from the generators in \eqref{eq:genpsizer} using Gau{\ss}ian elimination.
In order to compute $\Pol(o_F \dcroc{X}^{\psi=0})_{\leq N}$ efficiently, we need to compute for $b\in \{b_0,\dots,b_N\}$ and $0\leq m\leq N$ the polynomials $c_{b,m}$. For $0\leq i\leq N$, let us write $b_i=b_{i,0}+b_{i,1}X+\dots b_{i,N}X^N$ and define the matrices
\[
	B=\begin{pmatrix}
		b_{0,0} & b_{0,1} & \dots & b_{0,N}\\
		b_{1,0} & b_{1,1} & \dots & b_{1,N}\\
		\vdots  & \ddots  &       &  \vdots\\
		b_{N,0} & b_{N,1} & \dots & b_{N,N}
	\end{pmatrix}\in \mathrm{Mat}_{N+1,N+1}(o_F),
\]
and
\[
	C=\begin{pmatrix}
		c_{0,0} & c_{0,1} & \dots & c_{0,N}\\
		c_{1,0} & c_{1,1} & \dots & c_{1,N}\\
		\vdots  & \ddots  &       &  \vdots\\
		c_{N,0} & c_{N,1} & \dots & c_{N,N}
	\end{pmatrix}\in \mathrm{Mat}_{N+1,N+1}(\Int).
\]
We have
\[
	B\cdot C= \begin{pmatrix}
		c_{b_0,0} & c_{b_0,1} & \dots & c_{b_0,N}\\
		c_{b_1,0} & c_{b_1,1} & \dots & c_{b_1,N}\\
		\vdots  & \ddots  &       &  \vdots\\
		c_{b_N,0} & c_{b_N,1} & \dots & c_{b_N,N}
	\end{pmatrix}\in \mathrm{Mat}_{N+1,N+1}(\Int).
\]
Hence, in order to compute $\Pol(o_F\dcroc{X}^{\psi=0})_{\leq N}$ it remains to compute the polynomials $c_{i,j}$ efficiently. For all $i \in \NN$, we have
\begin{equation}
\label{eq:defcij}
	[a](X)^i = \sum_{j\geq i}c_{i,j}(a)X^j=\exp_{\LT}(a\cdot\log_{\LT}(X))^i.
\end{equation}
Let us denote by $D=\left( \langle \log_{\LT}(X)^j \mid X^k \rangle \right)_{0\leq j,k \leq N}$
the (truncated) Carleman matrix of $\log_{\LT}$. Then \eqref{eq:defcij} can be re-written as the matrix identity
\[
	D\cdot C=\begin{pmatrix}
	1 &   &  &		\\
	  & a &  &		\\
	  &   & \ddots &\\
	  &   &  & a^N
	\end{pmatrix}\cdot D,
\]
and we get $C=D^{-1} \mathrm{diag}(1,a,\dots, a^N)D$. For the computation of the Carleman matrix of $\log_{\LT}$, we have the following efficient recursive formula for the coefficients of $\log_{\LT}(X)$. Write $\log_{\LT}(X) = \sum_{k \geq 1} h_k X^k$ with $h_1 = 1$.

\begin{lemm}
\label{logcoeff}
We have $h_n  = 1/(p - p^n) \cdot \sum_{i = 1}^{\lfloor n/q \rfloor}  h_j \binom{j}{i} p^{j-i}$ where $j=n-i(q-1)$.
\end{lemm}

\begin{proof}
We have $\log_{\LT}(pX+X^q) = p \log_{\LT}(X)$. We can expand  $\log_{\LT}(pX+X^q)$ as
\[ \log_{\LT}(pX+X^q) = \log_{\LT}(pX) + \sum_{i \geq 1} X^{qi} \log_{\LT}^{[i]}(pX), \]
where $\log_{\LT}^{[i]}$ denotes the $i$th Hasse derivative of $\log_{\LT}$.
Computing the coefficient of $X^n$ on each side of $p \log_{\LT}(X) - \log_{\LT}(pX) = \sum_{i \geq 1} X^{qi} \log_{\LT}^{[i]}(pX)$, we get 
\[ (p - p^n) h_n = \sum_{i \geq 1} \langle \log_{\LT}^{[i]}(pX) \mid X^{n-qi} \rangle = \sum_{i = 1}^{\lfloor n/q \rfloor}  h_j \binom{j}{i} p^{j-i}, \text{ where $j=n-i(q-1)$}. \qedhere \]
\end{proof}

Fix a positive integer $N$. We now describe an algorithm for computing $s_0(n)$ for all $n\leq N$ (it returns $-1$ if $s_0(n)>N$).
\begin{enumerate}
	\item[(1)] Compute the Carleman matrix $D$ of $\log_{\LT}$, see Lemma \ref{logcoeff}.
	\item[(2)] Compute $(c_{i,j}(a))_{i,j}=C=D^{-1} \mathrm{diag}(1,\dots, a^N) D$.
	\item[(3)] Compute a basis $b_0,\dots, b_N$ of $o_F\dcroc{X}^{\psi=0}$ modulo $X^{N+1}$, see Lemma \ref{ofpsizer}, and store the coefficients of $b_0,\dots,b_N$ in a matrix $B$.
	\item[(4)] The matrix $B\cdot C=(c_{b_i,m})_{i,m}$ contains generators of $\Pol(o_F\dcroc{X}^{\psi=0})_{\leq N}$, see Corollary \ref{ofpsizer-coro}. 
	\item[(5)] For each $n$ with $0\leq n\leq N$ let $s_0(n)$ be the smallest $s\in\{1,\dots,N\}$ such that the $o_F$-module spanned by $c_{b_i,m}$ with $0\leq i,m\leq s$ contains a polynomial of leading coefficient $-w_q(n)$. If there is no $s$ with this property set $s_0(n)=-1$.
\end{enumerate}
For an implementation of this algorithm in SageMath, see Appendix A. The code is adapted from a similar program by Crisan and Yang, see the appendix of \cite{AB24}.

Here are some results, running SageMath 10.5 \cite{Sag25} on an M1 iMac.

\begin{enumerate}
\item For $p=2$ and $N=800$ and precision $6000$ we find that $s_0(n)$ is finite for $n \leq 206$

\begin{figure}[H]
\includegraphics[width=0.8 \textwidth]{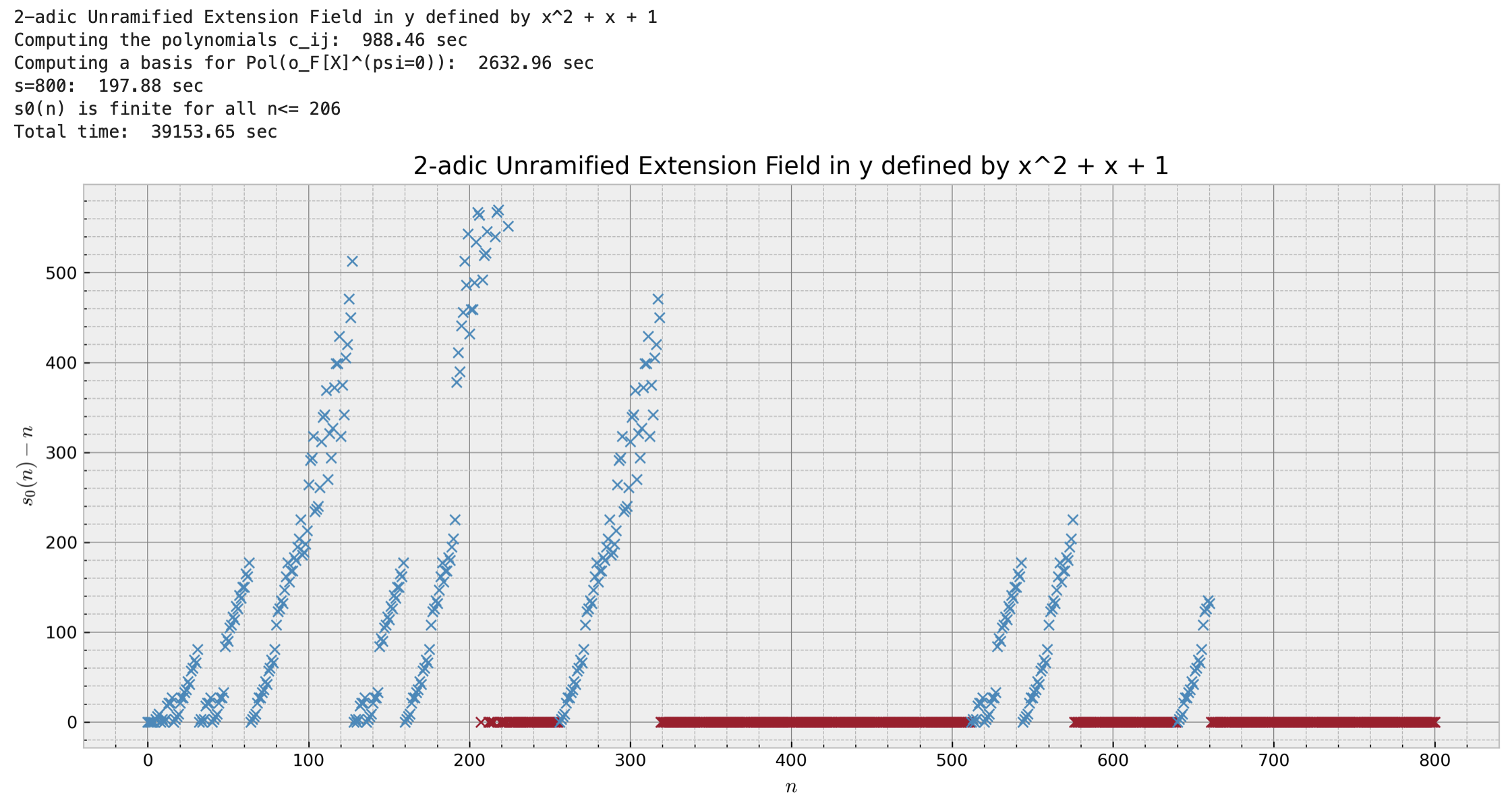}
\caption{Plot of $s_0(n)-n$ for $p=2$ and $N=800$. Red points are the $n$'s for which $s_0(n)=-1$.}
\end{figure}

\item For $p=3$ and $N=800$ and precision $6000$ we find that $s_0(n)$ is finite for $n \leq 226$

\begin{figure}[H]
\includegraphics[width=0.8 \textwidth]{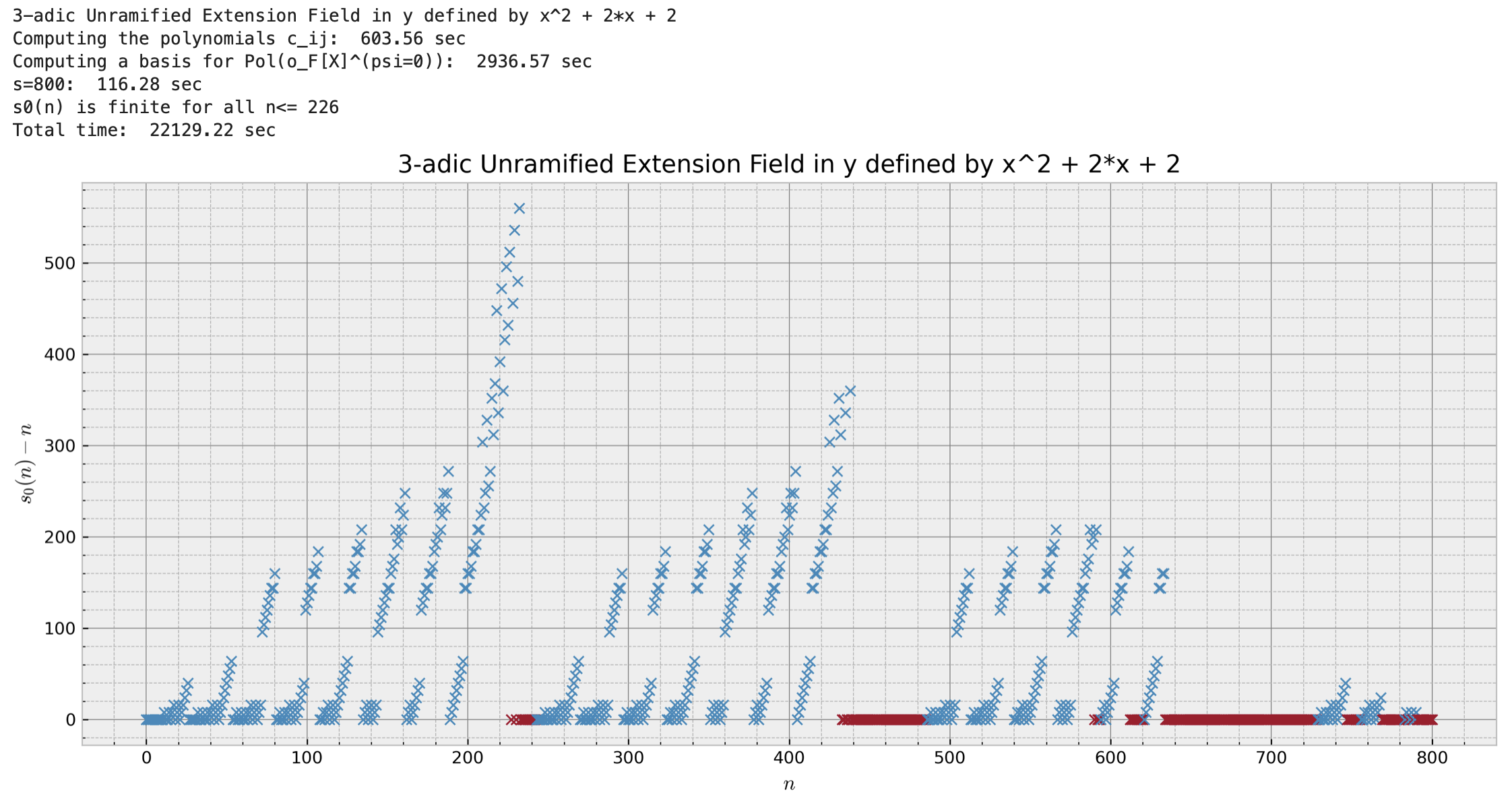}
\caption{Plot of $s_0(n)-n$ for $p=3$ and $N=800$. Red points are the $n$'s for which $s_0(n)=-1$.}
\end{figure}
\end{enumerate}

\appendix

\section{SageMath code}

\definecolor{sageOliveGreen}{rgb}{0,0.4,0}
\definecolor{sageTeal}{rgb}{0,0.4,0.4}
\definecolor{sageGray}{rgb}{0.5,0.5,0.6}
\definecolor{sageDarkRed}{rgb}{0.5,0,0.1}
\definecolor{sageWhite}{rgb}{1,1,1}
\lstdefinestyle{SAGEcode}{
    backgroundcolor=\color{sageWhite},
    commentstyle=\color{sageTeal},
    keywordstyle=\color{sageOliveGreen},
    numberstyle=\tiny\color{sageGray},
    stringstyle=\color{sageDarkRed},
    basicstyle=\ttfamily\tiny,
    breaklines=true
}
\lstset{style=SAGEcode}
\begin{lstlisting}[language=Python]
p=3   # prime number 
N=120   # cutoff 
precision = 1000   # p-adic precision

q=p^2   # Degree of the residue field

# Python imports
from time import process_time
import matplotlib.pyplot as plt
import numpy as np

# Definitions
import sage.rings.padics.padic_extension_generic

# Define the p-adic field, its ring of integers and the valuation function v
o_F.<y> = Zq(ZZ(q),prec=precision)
F.<y> = Qq(ZZ(q),prec=precision)
v = F.valuation()
print(F)

# Define the generator of the unique maximal ideal in o_F. 
Pi = o_F(p)

# Do linear algebra over the ring of polynomials F[X]
# in one variable X with coefficients in the field F:
F_X.<X> = F[]
F_Y.<Y> = F[]

# Define the rings of power series over F and o_F and the Frobenius lift phi=pT+T^q.
F_T.<T>=PowerSeriesRing(F,default_prec=N)
phi_F=Pi*T+T^q
o_F_tt.<t>=PowerSeriesRing(o_F,default_prec=N)
phi=o_F_tt(phi_F(t))

def LogLT(N):
    # Computes the logarithm of the Lubin-Tate formal group law with phi(T)=Pi*T+T^q up to precision O(T^(N+1))
    h=[0,1]
    for n in range(2,N+1):
        h.append(sum([F(h[n-q*i+i]*binomial(n-q*i+i,i)*p^(n-q*i)) for i in range(1,floor((n)/q)+1)])/(p-p^n))
    log_LT= sum([h[n]*T^n for n in range(1,N)])+O(T^(N+1))
    return log_LT

def LogCarlemanMatrix(N):
    # Computes the Carleman matrix of log_LT which is used for computing c_(i,j)
    CMat = matrix(F,N+1,N+1)

    # Compute the logarithm of the Lubin-Tate formal group law
    log_LT=LogLT(N)
    
    log_LT_j=F_T(1)
    for j in range(0,N+1):
        # Stores the coefficients of the power series log_LT(T)^j in row j of the matrix CMat
        log_LT_dict_j=log_LT_j.dict()
        l=[0 for _ in range(0,N+1)]
        for key in log_LT_dict_j.keys():
            l[key]=log_LT_dict_j[key]
        for i in range(0,len(l)):
            CMat[j,i]=l[i]
        log_LT_j=log_LT_j*log_LT+O(T^(N+1))
    return CMat

def BMatrix(N):
    # BMatrix returns a square matrix B of size (N+1)x(N+1) 
    # such that B[i,0]+X*B[i,1]+...+X^N*B[i,N] for i=0...N forms 
    # an o_F-basis of the degree <=N part of o_F[X]^(psi=0) 

    # Create a list (genList) of all generators of o_F[X]^(psi=0) of degree at most N
    genList=[]
    for i in range(0,q-1):
        phi_j=o_F_tt(1)
        for j in range(0,N-i+1):
            if i==0:
                genList.append((Pi*t^(q-1)-(1-q))*phi_j+O(t^(N+1)))
            else:
                genList.append(t^i*phi_j+O(t^(N+1)))
            phi_j=phi_j*phi
    
    # Store the coefficients of the polynomials in genList in the matrix B
    B=matrix(F,len(genList),N+1)
    i=0
    for gen in genList:
        l=gen.list()
        for j in range(0,len(l)):
            B[i,j]=l[j]
        i=i+1

    # Perform Gaussian elimination
    i0 = 0
    for k in range(B.ncols()):
        valuation_row_pairs = [ 
            (v(B[i,k]), i) for i in range(i0, B.nrows()) if B[i,k] != 0]
    
        if not valuation_row_pairs:
            raise ValueError("B is not full-rank")
        minv, i_minv = min(valuation_row_pairs)
    
        # Swap the row of minimum valuation with the first bad row
        B[i0, :], B[i_minv, :] = B[i_minv, :], B[i0, :]
    
        # Divide the top row by a unit in o_F
        u = B[i0, k] / Pi^int(v(B[i0, k]))
        B[i0, :] /= u
    
        # Cleave through the other rows
        for i in range(i0 + 1, B.nrows()):
            if v(B[i, k]) >= v(B[i0, k]):
                B[i, :] -= B[i, k]/B[i0, k] * B[i0, :]
    
        i0 += 1
        
    # Return the first B.ncols() rows of the matrix B
    return B[0:B.ncols(), :]

def CMatrix(N, D=None):
    # Computes a matrix containing the polynomials c_(ij) using the Carleman matrix of log_LT
    if D is None:
        D = LogCarlemanMatrix(N)

    # Define a diagonal matrix:
    Diag = matrix(F_X, N+1,N+1, lambda x,y: kronecker_delta(x,y) * X^x)

    # Compute the inverse of D:
    S = D.inverse()
    
    # Compute the matrix C:
    C = S * Diag * D
    
    return C

def w_q(n):
    return (n - sum(n.digits(base=q))) / (q-1)

def compute_s(N, filename=None):
    
    t_start = process_time()
    D = LogCarlemanMatrix(N)
    C = CMatrix(N, D)
    t_end = process_time()
    print(f"Computing the polynomials c_ij: {t_end-t_start : .2f} sec")
    
    t_start = process_time()
    # BPsi contains a basis of o_F[X]^(psi=0)_{<=N} of degree <=N
    BPsi=BMatrix(N)
    # Tau contains a basis of Pol(o_F[X]^{psi=0}_{<=N}) 
    Tau=BPsi*C
    t_end = process_time()
    print(f"Computing a basis for Pol(o_F[X]^(psi=0)): {t_end-t_start : .2f} sec")

    # s0_s[n] will store the minimal degree such that Int_n is contained in Pol(o_F[X]^{psi=0}_{<=s0_s[n]})
    s0_s = [-1 for _ in range(N+1)]
    
    B_old = Matrix(0,0)
    d = 0
    for s in range(N+1):
        t_start = process_time()
        
        # 1. Use the non-zero rows from previous calculations
        # 2. Add a 0 column to its left
        # 3. Add rows corresponding to entries from the j_th column of Tau_a
        B = Matrix(F, 2*s-d+1, s-d+1)
        B[0,0] = 1
        B[1:s-d+1, 1:] = B_old
        for i in [0 .. s-1]:
            coeffs = Tau[i, s].list()
            B[s-d+1+i, B.ncols()-len(coeffs)+d:] = vector(F, reversed(coeffs[d:]))
        
        # Perform Gaussian elimination
        i0 = 0
        ks = []
        for k in range(B.ncols()):
            valuation_row_pairs = [ 
                (v(B[i,k]), i) for i in range(i0, B.nrows()) if B[i,k] != 0]

            if not valuation_row_pairs:
                raise ValueError("B is not full-rank")
            minv, i_minv = min(valuation_row_pairs)
            ks.append(k)

            # Swap the row of minimum valuation with the first bad row
            B[i0, :], B[i_minv, :] = B[i_minv, :], B[i0, :]

            # Divide the top row by a unit in o_F
            u = B[i0, k] / Pi^int(v(B[i0, k]))
            B[i0, :] /= u

            # Cleave through the other rows
            for i in range(i0 + 1, B.nrows()):
                if v(B[i, k]) >= v(B[i0, k]):
                    B[i, :] -= B[i, k]/B[i0, k] * B[i0, :]

            i0 += 1
        
        d_is_updated = False
        for b in [d .. s]:
            n=b
            # if the valuation of the leading coefficient is for the first time -w_q then store s in s0_s
            if v(B[s-b, s-b]) == -w_q(n): 
                if s0_s[b] == -1:
                    s0_s[b] = s
            else:
                if not d_is_updated:
                    d = b
                    d_is_updated = True
        B_old = B[:s-d+1, :s-d+1]
        
        t_end = process_time()
        print(f"s={s}: {t_end-t_start : .2f} sec", end='\r')
        if filename is not None:
            with open(filename, 'w') as f:
                f.write("n,s0\n")
                for n, s0 in enumerate(s0_s):
                    f.write(f"{n},{s0}\n")
    print()

    # plot the result 
    plt.style.use('bmh')
    fig = plt.figure(figsize=(15,6), dpi=300)
    for n, s0 in enumerate(s0_s):
        if s0 != -1:
            plt.plot(n, s0-n, 'x', c='C0')
        else:
            plt.plot(n, 0, 'x', c='C1')
    plt.xlabel(r"$n$")
    plt.ylabel("$s_0(n) - n$")
    plt.title(str(F))
    plt.minorticks_on()
    plt.grid(which='both')
    plt.grid(which='major', linestyle='-', c='grey')
    
    return s0_s, fig

t_start_tot = process_time()
s0_s = compute_s(N);
print("s0(n) is finite for all n<=",(s0_s[0]+[-1]).index(-1)-1)
t_end_tot = process_time()
print(f"Total time: {t_end_tot-t_start_tot : .2f} sec")
\end{lstlisting}

\end{document}